\documentclass[11pt,a4paper]{amsart}
\usepackage{amsmath,amsfonts,amssymb,amsthm,amscd}
\usepackage[english]{babel}
\usepackage[mathscr]{eucal}
\usepackage{mathrsfs}
\usepackage[latin1]{inputenc}
\usepackage[pdftex,bookmarks=true,colorlinks=,pdfstartview=FitH]{hyperref}
\usepackage{ dsfont }
\usepackage{mathrsfs}
\usepackage{wasysym}
\usepackage{graphicx}
\usepackage{epstopdf}
\DeclareMathOperator{\dom}{dom}
\DeclareGraphicsExtensions{.pdf,.png,.jpg}
\everymath{\displaystyle}
\usepackage{graphicx,bm,bbm,cite}
\usepackage{epsfig}
\usepackage{latexsym}
\usepackage{url}
\usepackage{color}
\usepackage{multirow}
\usepackage{hyperref}
\usepackage{yfonts}
\usepackage{mathbbol}
\usepackage[inline,shortlabels]{enumitem}

\input{comment.sty}

\numberwithin{equation}{section}

\newcommand{\R}{{\Bbb R}}

\newcommand{\N}{{\Bbb N}}

\newcommand{\Ac}{{\mathcal A}}

\newcommand{\Fc}{{\mathcal F}}
\newcommand{\Gc}{{\mathcal G}}

\newcommand{\Jc}{{\mathcal J}}
\newcommand{\Kc}{{\mathcal K}}
\newcommand{\Lc}{{\mathcal L}}

\newcommand{\Oc}{{\mathcal O}}

\newcommand{\Rc}{{\mathcal R}}

\newcommand{\Tc}{{\mathcal T}}

\def\Es{{\mathscr{E}}}

\def\Ss{{\mathscr{S}}}

\newcommand{\supp}{\hbox{{\rm supp}}\,}

\newcommand{\GL}{\operatorname{GL\,}}
\def\1{\mathds{1}}

\newcommand{\vol}{\hbox{{\rm vol}}\,}
\newcommand{\Iinfty}[1]{I_{#1}^\infty}

\def\eps{{\varepsilon}}
\def\epi{{\rm epi}}

\def\raw{\rightarrow}

\def\vol{{\rm vol}}

\def\intt{{\rm int}}
\def\supp{{\rm supp}}
\def\cl{{\rm cl}}
\def\be{\begin{equation}}
\def\ee{\end{equation}}

\def\cvx{{\rm Cvx}}
\def\sc{{\rm SC}}  

\newtheorem{theorem}{Theorem}[section]
\newtheorem{proposition}[theorem]{Proposition}
\newtheorem{lemma}[theorem]{Lemma}

\theoremstyle{definition}
\newtheorem{definition}[theorem]{Definition}

\theoremstyle{remark}
\newtheorem{remark}[theorem]{Remark}

\begin{document}


\title{ The Existence of Extremizers of Blaschke-Santal\'o Type Inequalities}
\author[B. Li]{ Ben Li}
\address{School of Mathematical Sciences\\
Tel Aviv University }
\email{liben@mail.tau.ac.il}


\maketitle

\begin{abstract} 
   We discuss  a topological structure on families of convex functions and then apply it to show the existence of extrimizers for the functional Santal\'{o} inequality with respect to polar transform and its reverse.

\end{abstract}
\thispagestyle{empty}
\renewcommand{\thefootnote}{}
\footnotetext{2010 \emph{Mathematics Subject Classification}: 26A51 \and 52A41 \and 46B10}
\footnotetext{\emph{Key words and phrases}: Polar transformation, Blaschke-Santal\'{o} inequality, log-concave function.
	
This publication is part of a project that has received funding from the European Research Council (ERC) under the European Union's Horizon 2020 research and innovation programme (grant agreement No. 770127).}
\renewcommand{\thefootnote}{\arabic{footnote}}
\setcounter{footnote}{0}

\section{Introduction}

The Legendre transform \( \Lc\) is a  classical and well known operation which is an isomorphim of the set of closed proper convex functions \(\cvx(\R^n)\). For a convex function \( \phi\in \cvx(\R^n)\), it is defined by 
\(     \Lc\phi(y)=\sup_x\{  \langle x,y\rangle-\phi(x)  \}.\)
It is the only involution on \( \cvx(\R^n)\)\cite{artsteinmilman2009}.

\par 
The polarity transform \( \mathcal{A}\) of \(  \cvx_0(\R^n)\), the class of non-negative convex functions vanishing at 0 (called geometric convex functions), is defined as follows\cite{artsteinmilman2011} 
\[
\mathcal{A}\phi(y)=\begin{cases}
\sup_{\{x\in \R^n; \phi(x)>0\}}\frac{\langle x,y\rangle-1}{\phi(x)} \text{   if  } 0\neq y\in \{ \phi^{-1}(0)\}^\circ\\
0 \text{  if  } y=0\\
+\infty \text{  if  } y\neq \{ \phi^{-1}(0)\}^\circ
\end{cases}
\]
It was shown in \cite{artsteinmilman2011} that the epi-graph of \( \Ac[\phi]\) is  the reflection of the polar set of the \( \epi(\phi)\).

\par 

One of the central results in convex geometry concerns the volume product of a convex body and its polar body, which is sometimes refered to as Mahler product. It asserts that there is a universal constant \(c>0\) such that for any o-symmetric convex body \(K\in \R^n\),
\be
c^n\cdot w_n^2\le \vol(K)\vol(K^\circ)\le w_n^2,         \label{bsinequality}
\ee
where \( w_n=\vol (B_2^n)\).
The right hand side inequality is called Blaschke-Santal\'o inequality and was proved by Santal\'o in \cite{santalo1949}. The left hand side inequality is called Bourgain-Milman inequality which was shown in \cite{bourgainmilman1987}. It should be noticed that the Mahler product is invariant under non-singular affine transformations. The optimal lower bound for the Mahler product is still kept in mystery, known as Mahler conjecture. It asks whether the optimal lower bound of (\ref{bsinequality}) achieves among all o-symmetric convex bodies at the so called Hanner polytopes (e.g., the high dimensional cube or the cross polytope) or achieves among all convex bodies with centroid o at simplices.  A recent break through was due to Iriyeh and Shibata \cite{Iriyeh3dmahler}.  They proved  this conjecture for 3-dimensional convex bodies (see also \cite{FHMRZ20193dmahler} for a simplified proof).
\par 
The functional counterparts of volume product corresponding to different ``polarity" are studied extensively in the past decades. Let \( \Tc\in \{ \Lc, \Ac\}\). We define \( P_\Tc(\cdot)\) on \( \cvx(\R^n)\) if \( \Tc= \Lc\), on \( \cvx_0(\R^n)\) if \( \Tc=\Ac\) as following 
\be 
  P_\Tc(\phi):= \int_{\R^n} e^{-\phi(x)}dx\int_{\R^n} e^{-\Tc \phi(x)}dx   \label{functionalLA}
\ee
for \( \phi\in \cvx(\R^n)\), respectively, \( \phi\in \cvx_0(\R^n)\). 
\par 

\par 
 In the case of  \( \Tc=\Ac\), we consider the functional \( P_\Ac(\cdot)\) on the subset of geometric convex functions \(\cvx_0^+(\R^n)= \{ \phi\in \cvx_0(\R^n): 0<\int_{\R^n} e^{-\phi(x)}dx<\infty\}\).    The Blaschke-Santal\'o type inequality  were proved in \cite{artsteinslomka2015}, it states that there exist universal numerical constants \(c,C>0\) such that for any  integrable geometric log-concave function \(f=e^{-\phi}\) with centroid at 0, 
\be
 c^nw_n^2\le P_\Ac(\phi)= \int_{\R^n} e^{-\phi(x)}dx\int_{\R^n} e^{-\Ac\phi(x)}dx  \le (n!w_n)^2(1+\frac{C}{n})\label{polarsantalo}
 \ee
		The left hand side inequality holds also without the assumption that the centroid of \(\phi\) is at 0. The authors remarked also that if a  maximizer exists, there must be a  rotationally invariant maximizer. But it was not clear whether the maximizer and the minimizer exist. The result of this short note will give a affirmative answer to this question. 
\begin{theorem}\label{mainthm1}
	
			The functional \( P_\Ac(\cdot)\) achieves maximum and minimum on the subset of even functions in \( \cvx_0^+(\R^n)\). Moreover,
	  it achieves maximum and minimum on the subset of functions in \( \cvx_0^+(\R^n)\) with centroid 0.

\end{theorem}

When \( \Tc=\Lc\),  the bounds for the value of the functional \( P_\Tc(\cdot)\) on the set of functions \( \cvx^+(\R^n)= \{ \phi\in \cvx(\R^n): 0<\int_{\R^n} e^{-\phi(x)}dx<\infty \}\) are estebalished and this  is refered to as the functional Blaschke-Santal\'o inequalities for log-concave functions \cite{klartagmilman2005}, i.e., there exists an absolute constant \(c>0\) such that for any integrable even log-concave function \( f= e^{-\phi}\)  
\be
\left(\frac{2\pi}{c}\right)^n\le P_\Lc(\phi)= \int_{\R^n} e^{-\phi(x)}dx\int_{\R^n} e^{-\Lc \phi(x)}dx   \le (2\pi)^n.    \label{legendresantalo}
\ee
The maximum for the functional \( P_\Lc(\cdot)\) on even convex functions was first observed by Ball \cite{ballphd} and later was extended to general case by Artstein-Avidan, Klartag and Milman \cite{AKM2004} and Fradelizi and Meyer \cite{FM2007}. In the latter two papers it is proved that the maximum of \( P_\Lc\) is achieved uniquely at Gaussians. The left hand side inequality of (\ref{legendresantalo}) is due to Klartag and Milman \cite{klartagmilman2005} and it is deduced from the Bourgain-Milman inequality. However, to the best of the author's knowledge, it is not clear whether the tight lower bound exists. 
The next theorem shows that the lower bound should also be saturated. 

\begin{theorem}\label{mainthm2}
	
	The functional \( P_\Lc(\cdot)\) achieves  minimum on the subset of even functions in \( \cvx^+(\R^n)\). Moreover, it achieves  minimum on the subset of functions in \( \cvx^+(\R^n)\) with centroid 0.

\end{theorem}

On \( \cvx_0(\R^n)\) another isomorphim which is called gauge transform, denoted by \( \mathcal{J}\),  was studied in \cite{artsteinmilman2011} which is given by 
\[
(\mathcal{J}\phi)(x) = \inf \left\{  r>0: \phi\left(\frac{x}{r}\right)\le \frac{1}{r}   \right\}.\]
A key observation  \cite{artsteinmilman2011}  is  that  \( \Ac\circ\Lc = \Lc \circ \Ac = \Jc\). The polar transform and the guage transform are the only order reversing involutions defined on \( \cvx_0\) \cite{artsteinmilman2011}.

\par 
Using the same method it can be shown that the functional \(P_\Jc(\phi)= \int e^{-\phi}\bigg/\int e^{-\Jc \phi}\) has a maximum and  a minimum on \( \cvx_0^+\). 
        In fact, a Blaschke-Santalo type inequality for \( \Jc\) transform can be established as well. It was obtained by D. Florentin and A. Segal very recently \cite{alexdan} that on \(\cvx_0^+\),
		\be
		  c_n^{-1}  \le P_\Jc(\phi)= \int e^{-\phi}\bigg/\int e^{-\Jc \phi}\le c_n
		 \ee
		 where the cosntant \( c_n\) depends on the dimension and is asymptotically \(n!\). 
		They are able to characterize the extremal occations. The maximizer  is a truncated norm which is truncated at  height asymptotically $\frac{1}{n}$.  The minimizer is simply the \(\Jc\) transform of the maximizer \cite{alexdan}. 

\par 
\vskip 3mm
For a convex function \( \phi: \R^n\to \R\cup \{\infty\}\), we denote its sub-level sets by
\( G_\phi(s)= \{ x\in \R^n: \phi(x)\le s\}, s\in \R\) and its centroid (with respect to Lebesgue measure) by 
\[c(\phi)=\frac{\int_{\R^n} xe^{-\phi(x)}dx}{\int_{\R^n} e^{-\phi(x)}dx}.
\]
The convex indicator function \( I_K(x)=0\) whenever \( x\in K\) and \( I_K(x)=\infty\) elsewhere. For two functions \( \phi\) and \( \psi\),  \( \phi\wedge\psi(x)= \min\{ \phi(x), \psi(x)\}; \phi\vee\psi(x)= \max\{ \phi(x), \psi(x)\}\).   
In this note we denote by \( \Fc\) the collection of closed sets, \( \Kc\) the collection of compact sets and \( \Gc\) the collection of open sets. All sets in these collections should be understood of subsets in  \( \R^n\) (or \(\R^{n+1}\)) equipped with the standard topology. 
\par 
\vskip 3mm
In section 2, we will briefly  recall the definition of \( \tau\)-topology defined on the set of functions, collect useful lemmae and derive the continuity of polar transform with respect to \( \tau\)-topology. Section 3 and 4 are devoted to the proof of the main theorems.  

\vskip 3mm
\noindent
{\bf Acknowledgement}
\par
\noindent
The author is very grateful to Prof. Artstein-Avidan, Dr. Slomka and Dr. Mussnig for insightful discussions and their patience. 
This publication is part of a project that has received funding from the European Research Council (ERC) under the European Union's  Horizon 2020 research and innovation programme (grant agreement No. 770127).

\section{Topology for Closed Sets and Functions}


\par 
\noindent
In this section we first describe the \(\tau\) topology for the set of closed sets in \( \R^n\). Some results presented here can be found in Chapter 4 of \cite{RockafellarWets1998} and Section III of \cite{dolecki1983}. Then we discuss \(\tau\)-topology for convex functions.
We also collect a few known results on functions which will be used repeatedly. Other known lemmae are introduced in the sections in which they are used.

\subsection{Topology of the Space of Closed Sets}

We start with defining semi-limit of sets. For a sequence of subsets \( \{ K_m\}_{m\in \N}\), 
\( \limsup_m K_m\) consists of  elements that are limits of elements contained in \(K_m\) for infinitely many \(m\) and \( \liminf_m K_m\) consists of  elements that are limits of elements contained in \(K_m\) for all but finitely many \(m\). That is, 
\begin{definition}[\cite{RockafellarWets1998}]
\begin{eqnarray*}
	 \limsup_{m\to \infty } K_m &:=& \{ x: \exists \{ m_j\} \text{ such that } x_{m_j}\in K_{m_j} \text{ and } x_{m_j}\to x \text{ as  } j\to \infty \};
\end{eqnarray*}
\[ \liminf_{m\to \infty}  K_m := \{ x:  \exists M \text{ so that  for } m>M, \exists x_m\in K_m \text{ and } x_m\to x \text{ as } m\to \infty  \}  .   \]


 We say that a sequence of sets \( K_m\) converges in \(\tau\)-topology to a set \(K\) , denoted by \( K_m\xrightarrow{\tau} K\),  if 
  \[ \limsup K_m\subset K \subset \liminf K_m.\]
\end{definition}
\begin{remark}
	(1) Note that for convex bodies, this topology is equivalent to the topology induced by Hausdorff metric \( d_H\) (see, e.g., Theorem  1.8.8 of  \cite{SchneiderBook}). It's also called \( \Gamma\)-topology in some literature. \\
	(2) The \( \tau\)-topology of closed sets is generated by the subbase of open sets 
	\be  \{ \Fc^K:  K \text{ is compact }\} \text{  and   } \{ \Fc_G: G \text{ is open }\}.     \label{subbaseoftau}           \ee 
	where \( \Fc\) is the collection of closed subsets of \( \R^n\) and for any set \( S\in \R^n\), 
	\[    \Fc^S:= \{ F\in \Fc: F\cap S= \emptyset\}, \ \ \ \ \Fc_S:= \{ F\in \Fc: F\cap S\neq \emptyset\}. \]
	(3) The \( \tau\)-topology of closed sets is metrizable (see, e.g., \cite{RobertsVarbergBook1973}). Thus compactness is equivalent to sequencial compactness under this topology. Moreover, with this topology the space of closed subsets of \(\R^n\) is compact \cite{dolecki1983}.
\end{remark}

\par 
Next lemma will be useful. 
\begin{lemma}[\cite{RockafellarWets1998}]\label{continuityofsetlimit}
For a sequence \( \{ C_m\}_m\) of convex sets of \( \R^n\), \( \liminf_{m\to \infty}C_m\) is convex, and so is \( \lim_{m\to \infty}C_m\) whenever it exists.
\end{lemma}

\subsection{Limit of Functions}
Now we describe \(\tau\)-topology for functions. 
\begin{definition}[\cite{RockafellarWets1998}]
	A sequence of functions \( \phi_m\) converges to \( \phi\) in \(\tau\)-topology if \[ \epi(\psi_m)\xrightarrow{\tau} \epi(\psi),\]
	written as \( \phi_m \xrightarrow{\tau}\phi\).  We say \( \phi= \tau\)-\(\liminf \phi_m \) if \(\epi(\phi)=\liminf \epi(\phi_m) \);
	\( \phi= \tau\)-\(\limsup \phi_m \) if \(\epi(\phi)=\limsup \epi(\phi_m) \). 
\end{definition} 
\par 
Note that 
\( I_{K_m}^\infty \xrightarrow{\tau} I_K^\infty \) if and only if \(    K_m\to K\). We shall see that the polar transform for geometric convex functions  is continuous with respect to \(\tau\)-topology. 
Before we prove this result, we collect a few well known results for functions associated with \( \tau\)-topology. 
\par 
\begin{lemma}[\cite{mosco1971}]\label{convlevelset}
	Suppose that \( \{\phi_m\}, \phi\in \cvx(\R^n)\) are such that \( \phi_m\xrightarrow{\tau} \phi \). Then the level sets \( G_{\phi_m}(t)\to G_\phi(t)\) for \( t\neq \inf \phi\).
\end{lemma}
%
\par 
\begin{lemma} [\cite{AKM2004}]\label{lemmafromAKM}
	Suppose that \( f_m=e^{-\phi_m}, f=e^{-\phi}, m=1,2,...\) are log-concave functions where \( \phi_m, \phi\) are convex functions satisfying  
	\( f_m\to f\) pointwise on a dense set of \( \R^n\). Then  
	\begin{enumerate}[(1)]
		\item \(\int f_m\to \int f\).
	   \item  \(\int xf_m\to \int xf\) if \( \int f<\infty\).
	   \item   \( \mathcal{L}(f_m)\to \mathcal{L}(f) \) locally uniformly on the interior of \( \supp(\Lc (f))\) where \(\Lc\) is the Legendre transform. 
		\end{enumerate}
\end{lemma}

\begin{lemma}[\cite{RockafellarWets1998}]\label{ptwconvergencedenseset}
	For any sequence \( \{\phi_m\}\) of convex functions on \( \R^n\) such that \( \phi_m\xrightarrow{\tau} \phi\), the function \( \phi\) is convex . If, moreover, \( dom(\phi)\) has nonempty interior,  the following three statements are equivalent: 
		\begin{enumerate}[(1)]
			\item \( \phi_m\xrightarrow{\tau} \phi\)
			\item there is a dense subset \( A\subset \R^n\) such that \( \phi_m(x)\to \phi(x) \) for all \(x\) in \(A\).
			\item \( \phi_m\) converges uniformly to \( \phi\) on every compact set \(C\) that does not contain a boundary point of \(dom(\phi)\).
		\end{enumerate}
\end{lemma}
\noindent 
\par 
\noindent 
\begin{lemma}[\cite{robert1974}]\label{robert1974:ptwint}
	If \( \phi_m\) and \(\phi\) are closed convex functions with \( \phi_m\xrightarrow{\tau} \phi\), then \( \phi_m\to \phi\) pointwise on \( \intt \dom(\phi)\).
\end{lemma}

Next proposition is well known. Here it readily follows from Lemma \ref{lemmafromAKM} and Lemma \ref{ptwconvergencedenseset} . 
\begin{proposition}[\cite{salinettiwets1977}]\label{contioflegendre}
	\( \phi_m \xrightarrow{\tau}\phi\) is equivalent to \( \Lc \phi_m \xrightarrow{\tau}  \Lc\phi\).
\end{proposition}
\par

Next lemma addresses the continuity of polarity operation on closed (not necessarily bounded) convex sets when the latter is equpped with \(\tau\)-topology, it was shown in \cite{mosco1971} (see also \cite{wijsman1966}).For completeness, we provide a proof.
 \begin{lemma}[\cite{mosco1971}]\label{convpolarset}
 	Let \( \{K_m\}_{m=1}^\infty, K\) are closed convex subsets in \( \R^n\) satisfying \( K_m\xrightarrow{\tau} K\). Then \( K_m^\circ\xrightarrow{\tau} K^\circ\).
 \end{lemma}
\begin{proof}
	With Proposition \ref{contioflegendre} and direct computation it is clear that 
	\begin{eqnarray*}
		K_m\xrightarrow{\tau} K \Leftrightarrow I_{K_n}^\infty \xrightarrow{\tau} I_K^\infty
		               \Leftrightarrow \Lc(I_{K_n}^\infty) \xrightarrow{\tau} \Lc(I_K^\infty)
		              \Leftrightarrow h_{K_m}\xrightarrow{\tau} h_K.
	\end{eqnarray*}
Since \( \inf_{x\in \R^n} h_K(x)<1\), by Lemma \ref{convlevelset}, 
\( G_{h_{K_m}}(1)\to G_{h_K}(1).\)
But \( G_{h_K}(1)= \{x: h_K(x)\le 1\}= \{ x: \|x\|_{K^\circ}\le 1\}= K^\circ\), we then have 
\( K_m^\circ \xrightarrow{\tau} K^\circ.\)
\end{proof}

Next lemma aims to compare sub-level sets of a convex function and its polar dual and Legendre dual. It can be found in \cite{RockafellarBook1970} (see also \cite{artsteinslomka2015} and \cite{FM2008}).
\begin{lemma}[\cite{RockafellarBook1970}]\label{comparisonoflevelsets}
	For \( \phi\in \cvx_0, s>0\), 
	\begin{enumerate}
		\item \(\{ x: \Ac \phi (x)\le s^{-1}\}   =  s^{-1} \{x: \Lc \phi(x)\le s\}\);
		\item \( \{ x: \phi(x)\le s\}^\circ  \subset s^{-1}\{x: \Lc \phi(x)\le s\}  \subset 2\{ x: \phi(x)\le s\}^\circ.  \)
	\end{enumerate}

\end{lemma}
Moreover, it follows immediately from  Lemma \ref{lemmafromAKM} and Lemma \ref{ptwconvergencedenseset} that  the centroid with respect to Lebesgue measure is continuous with respect to \( \tau\)-topology. 
\begin{proposition}\label{contiofcentroid}
	If \( \phi_m \xrightarrow{\tau}\phi\) and \( 0<\int e^{-\phi}<\infty\), then \( c(\phi_m) \to c(\phi)\). 
\end{proposition}
\par

\section{The space of Geometric Convex Functions and the Extremal of Polar Blaschke-Santal\'{o}  inequality}

\subsection{The  space of Geometric Convex Functions}
\par 
We consider the class of  closed proper convex functions on \(\R^n\). A convex function $\phi: \R^n \rightarrow\R \cup \{ \infty \}$ is proper if \( \phi(x)<\infty\) for at least one point and \(\phi(x)>-\infty\) for every \(x\); it is closed if \(\epi(\phi)\) is a closed set.  Note that for a proper convex function, closedness is the same as lower semi-continuity.  For basic properties of convex functions, we refer the reader to \cite{RockafellarBook1970}. We donote by \( \dom(\phi)\) the set \( \{ x\in \R^n: \phi(x)<\infty\}\). 

\par 
Let \( \sc(\R^n)\) be the set of lower semi-continuous functions, \( \cvx(\R^n)\)  the set of closed convex functions and \( \cvx_0(\R^n)\) the set of nonegative closed convex function vanishing at the origin. Clearly, \( \cvx_0(\R^n)\subset \cvx(\R^n)\subset \sc(\R^n)\).

In the following, we see that with the \(\tau\)-topology introduced in the previous chapter, \( \sc(\R^n)\) and \( \cvx_0(\R^n)\)  form compact topological spaces. 
\begin{proposition}[\cite{dolecki1983}]\label{dolecki1983}
   The topological space \( (\sc(\R^n), \tau)\) is compact. 
\end{proposition}

\begin{proposition}\label{cvx0compact}
	The topological subspace \( (\cvx_0(\R^n), \tau)\) is compact. 
\end{proposition}
\begin{proof}
For \((\cvx_0(\R^n), \tau)\subset (\sc(\R^n), \tau)\), we shall show that \[ \{ F\in \Fc: F=\epi(\phi) ,\phi\in \cvx_0(\R^n)\}\]
is closed in \( \tau\) topology.  Note that for any open set \(G\subset \R^n \), \( \Fc^{G\times (-\infty,0)}\) is closed in \( \tau\) topology, so is \( \Fc_{ \{0\}\times [0,a]}\) for any \( a\ge 0\). Now observe that 
\begin{eqnarray*}
	  & &\{ F: F=\epi(\phi) ,\phi\in \cvx_0(\R^n)\}=\left( \bigcap_{G\in \Gc} \Fc^{G\times (-\infty,0)}\right) \bigcap\left( \bigcap_{a\ge 0}\Fc_{ \{0\}\times [0,a]}\right)\\
	  & &\bigcap \{ F\in \Fc: F=\epi(\phi) ,\phi\in \sc(\R^n)\}\bigcap \{ \text{convex sets}\}.  
	\end{eqnarray*}
The collection \(\{ F\in \Fc: F=\epi(\phi) ,\phi\in \sc(\R^n)\}\) is closed by Proposition \ref{dolecki1983} and  the collection of convex sets  is closed by Lemma \ref{continuityofsetlimit}. Hence 
the right hand side of the above equation is closed in \(\tau\) topology for it is an intersection of closed sets in \( \tau\) topology. Therefore the propostion follows. 
\end{proof}
We will use the classical John's theorem\cite{john1948} (see, e.g., \cite{GardnerBook2006}) which we recall now. 
\begin{theorem}[John's Theorem]\label{johnthm}
	For each central symmetric convex body \(K\in \R^n\), there is a unique ellipsoid \(E\) with the same center for which \( E\subset K\subset \sqrt{n}E\). For each convex body \(K\in \R^n\), there is a unique ellipsoid \(E\) with  center \(b\in K\) for which \( b+E\subset K\subset b+nE\). 
\end{theorem}
\par 
\noindent 
We are concerned with, in this section,  the behavior of the functional \( P_\Ac(\cdot)  \)  (cf.  (\ref{functionalLA})) on a  subset of  geometric convex functions.
We denote the set of non-degenerate geometric convex functions by 
\[  \cvx_{0}^+:= \left\{ \phi\in \cvx_0:  0< \int e^{-\phi} <\infty\right\},  \]
and the set of non-degenerate geometric convex functions with centroid 0 by 
\[  \cvx_{0,c}^+:= \left\{ \phi\in \cvx_0:  0< \int e^{-\phi} <\infty, c(\phi)=0 \right\}. \]
\par 
\noindent 

For heuristical reasons we treat firstly the case of even functions. The general case is more involved and will be treated thereafter. 
\par 
Let \(\Ss_e\) be  the set of functions 
\be
 \mathscr{S}_e: = \left\{ \phi\in \cvx_0: \phi(x)=\phi(-x), \|\cdot\|_2\wedge I_{B_2^n}^\infty  \le \phi \le     1+ I_{B_2^n/\sqrt{n}}^\infty  \right\}.     \label{defnsetevengcf}
 \ee
Equivalently, by the convexity, \( \phi\in \mathscr S_e\) if and only if \( \phi\in \cvx_0\) and \( B_2^n/\sqrt{n}\subset G_{\phi}(1)\subset B_2^n\). Clearly, 
\( \mathscr S_e\) is a subset of \( \cvx_0^+\), as any function in \( \mathscr S_e\) has finite positive volume, i.e., 
\[ \int e^{-\phi(x)}dx  \le \int e^{-\|\cdot\|\wedge I_{B_2^n}^\infty(x) }   dx \le \vol(B_2^n)(n\cdot n!+1)<\infty,\]
and 
\[ e^{-1} n^{-n/2}\vol(B_2^n)=\int e^{-1-I_{B_2^n/\sqrt{n}}^\infty(x)} dx \le \int e^{-\phi(x)}dx.  \]

\par 

By John's theorem \ref{johnthm}, any function in \( \cvx_0^+\) can be brought into \( \mathscr S\) by some \( T\in \GL(n)\). Indeed, given an even function \( \phi\in \cvx_0^+\), its sub-level set \( G_{\phi}(1)=\{ x: \phi(x)\le 1\}\) is an o-symmetric convex body. The John's theorem assures that  there is \( T\in \GL(n)\) such that \( B_2^n/\sqrt{n}\subset T(G_{\phi}(1))\subset B_2^n\). Thus \( T^{-1}\phi\in \mathscr{S}_e\). 

\par 

\begin{lemma}\label{cptcvx0even}
	\( \mathscr S_e\) is a closed subset of \( \cvx_0(\R^n)\) and hence it is compact. 
\end{lemma}

\begin{proof}
	Suppose \( \{\phi_m\}_m\) is a sequence of function in \( \mathscr S_e\) such that \( \phi_m\xrightarrow{\tau} \phi\). The compactness of the space \( \cvx_0(\R^n)\) guarantees that \( \phi\in \cvx_0(\R^n)\).  By Lemma \ref{convlevelset}, \( G_{\phi_m}(1)\to G_\phi(1)\) in Hausdorff topology. It follows that \( B_2^n/\sqrt{n}\subset G_{\phi}(1)\subset B_2^n\). As for the eveness, since \( \phi\) is a closed convex function, it is enough to check on values in the interior of its domain. Hence for any \( x\in \intt(\dom(\phi))\), by Lemma \ref{robert1974:ptwint},  we have  \( \phi(x)= \lim_{m\to \infty} \phi_m(x) = \lim_{m\to \infty} \phi_m(-x)=\phi(-x)\).
\end{proof}

\par 
\noindent 
Next we consider the non-degenerate geometric convex functions, not necessarily even, but with centroid at the origin.
\par 
In this case we take a look at the following subset of functions:
\be 
\mathscr{S}_1:= \bigcup_{0\le t \le  n}\{ \phi\in \cvx_{0}:   te_1+B_2^n\subset G_\phi(1)\subset 2nB_2^n \}
\ee
 and  
\be 
\mathscr{S}_{1,c}:= \{ \phi\in \mathscr{S}_1:\int_{\R^n} x e^{-\phi(x)}dx=0 \}= \Ss_1\cap \cvx_{0,c}^+,
\ee
where \( e_1=(1,0,0,\cdots, 0)\).

\par 
Observe that any geometric convex function \( \phi\in \cvx_{0,c}\)  can be brought into \( \mathscr{S}_{1,c}\) by a linear transformation. Indeed, \( \phi\in \cvx_{0,c}\)  implies that \( G_\phi(1)\) is a convex body with \( 0\in \intt(G_\phi(1))\). By the John theorem  \ref{johnthm}, there are \( a\in \R^n, T\in \GL(n)\) such that 
\be 
a+TB_2^n\subset G_\phi(1)\subset a+ nTB_2^n. \label{nonsymjohn}
\ee
\(a\) is the center of the John ellipsoid of \(G_\phi(1)\). Applying \(T^{-1}\) yields 
\be
T^{-1}a+B_2^n\subset T^{-1}G_\phi(1)\subset T^{-1}a+ nB_2^n. \label{nonsymjohn1}
\ee
Appealing to  an appropriate rotation \( O\in O(n)\),  one has 
\be 
te_1+B_2^n \subset OT^{-1}G_\phi(1)\subset te_1+nB_2^n
\ee
where \( e_1=(1,0,...,0)\) and \( te_1= O T^{-1}a, t\ge 0\). Since \( 0\) is still in the interior of \( OT^{-1}G_\phi(1)\), \( t< n\). It follows that \( OT^{-1}G_\phi(1)\subset 2nB_2^n\). 
Since the centroid does not change by a linear transformation, \(TO^{-1}\phi\in \mathscr{S}_{1,c}\).
\par 
Now we firstly show that \( \mathscr S_1\) is  closed. 
\par 
\begin{lemma} \label{compactnesss1}
	\(\mathscr{S}_1\) is closed and hence compact.
\end{lemma}
\begin{proof}
	We shall prove here that the complement of \( \Ss_1\), \( \mathscr{S}_1^c\),  is open in \( (\cvx_0, \tau)\).
	Let 
	\be 
	\mathscr{E}_0=\{ F\in \Fc: F= \epi(\phi) \text{ for some } \phi\in \cvx_0\}.
	\ee
	\be 
	\mathscr{E}_1=\{ F\in \Fc: F= \epi(\phi) \text{ for some } \phi\in \mathscr{S}_1\}.
	\ee
	Elements in \( \mathscr{E}_0\) (resp. \( \mathscr{E}_1\)) is  one-to-one correspondent to elements in \( \cvx_0\) (resp.  \( \mathscr{S}_1\)).
	
	\par 
	Let \( F\in \mathscr{E}_0\setminus \mathscr{E}_1\).  Then either of the following two cases (or both) will happen. 
	
	Case 1. Let \( G=(2nB_2^n)^c\times (0,1)\)and  \( F\cap G\neq \emptyset\). Since \(G\) is an open set, \( \Fc_G\) is open in \( \tau\) topology.  Clearly,  \( \Fc_G\cap \mathscr{E}_1=\emptyset\). Thus \( \Fc_G\cap \intt(\Es_0) \) is an open neighborhood containing \(F\) and contained in \(\mathscr{E}_0\setminus \mathscr{E}_1\).
	
	Case 2. \(F\cap \{x_{n+1}=1\}\subset 2nB_2^n\) does not contain any    \( {\intt} ( te_1+ B_2^n)\), \( t\in [0,n]\).  We claim that there are \(\{x_t\in \R^n\}_{0\le t\le  n}\), \( \{0<a_t<1\}_{0\le t\le  n}\) and \( O_t\in \Gc, 0\le t\le n, G'\in \Gc \)(independent of \(t\)) such that 
		\be 
		x_t\in O_t,\ \ \  O_t\cap te_1+B_2^n\neq \emptyset,\ \ \ O_t\times (a_t,1]\subset (G')^c\subset F^c\label{ben}
		\ee
		for all \( 0\le t\le n\). 
		
		\par 
		To see this, we need a distance function defined for each closed set. For closed sets \( K, F\)   
		define 
		\be
		 d_F(K)= \max_{x\in K}d(x,F)=\max_{x\in K}\inf_{y\in F}\|x-y\|_2.\label{distancebtwsets}
		 \ee
		Now observe that \( d_F(K)\) is lower semi-continuous, i.e., as \( K_m\to K\), 
		\be
		 \liminf_{m\to \infty} d_F(K_m)\ge d_F(K).\label{lslofsets}
		\ee
		Indeed, if  \( K_m\to K\), for \( \forall x \in K\),
		by definition of set convergence  there is a sequence \( \{x_m\}_m\) such that \( x_m\in K_m\) and \( x_m\to x\). Thus 
		\begin{eqnarray*}
		  d(x,F)&=& \inf_{y\in F}\|x-y\|_2
		  \le \inf_{y\in F} \|x-x_m\|_2+\|x_m-y\|_2
		  = \|x-x_m\|_2+ d(x_m,F) 
		\end{eqnarray*} 
	    Taking \( \liminf\) with respect to \(m\) on both sides yields 
	    \[ d(x,F)\le \liminf_{m\to \infty}d(x_m, F).\]
	    Then by definition of \( d_F(\cdot)\) (\ref{distancebtwsets}), 
	    \[
	    d(x,F)\le \liminf_{m\to \infty}d(x_m, F)\le \liminf_{m\to \infty} d_F(K_m).
	    \]
		Taking maximum over \(x\in K\) on the left hand side gives (\ref{lslofsets}). 
		\par 
		In this  case, 
		\(F_1:=F\cap \{x_{n+1}=1\}\) not containing any   \( {\intt} ( te_1+ B_2^n), 0\le t\le n\) is equivalent to the statement 
		\be
		   d_{F_1}(te_1+ B_2^n)>0, 0\le t\le n.\label{unifpositivedistance}
		\ee
		Thus we have \( \inf_{0\le t\le n}d_{F_1}(te_1+ B_2^n)\ge 0\).
		Now we show that \(\inf_{0\le t\le n}d_{F_1}(te_1+ B_2^n)> 0 \). Suppose that \( \inf_{0\le t\le n}d_{F_1}(te_1+ B_2^n)= 0\). Then there exists an infinite sequence \( \{t_m\}_m\) such that 
		\[ \lim_{m\to \infty}d_{F_1}(t_me_1+ B_2^n)= 0.  \]
		Moreover, since \( \{ t_m\}\subset [0,n]\), there is a subsequence \( \{ t_{m_j}\}_j\) satisfying \( \lim_{j\to \infty } t_{m_j}=t_0\in [0,n]\) and 
		\be
		 \lim_{j\to \infty}d_{F_1}(t_{m_j}e_1+ B_2^n)= 0.
		 \label{lslimita} 
		\ee 
		On the other hand, since obviously \(t_{m_j}e_1+ B_2^n\to t_0e_1+ B_2^n \) and  (\ref{lslofsets}),
		\be 
		\liminf_{j\to \infty} d_{F_1}(t_{m_j}e_1+ B_2^n)\ge d_{F_1}(t_0e_1+ B_2^n)
		\label{lslimitb}
		\ee
		Combining (\ref{lslimita}), (\ref{lslimitb}) and non-negativity of the function \(d_{F_1}(\cdot)\), we have \[ d_{F_1}(t_0e_1+ B_2^n)=0,   \]
		which is a contradiction to (\ref{unifpositivedistance}). Therefore we have shown that \(\inf_{0\le t\le n}d_{F_1}(te_1+ B_2^n)> 0 \). 
		
		\par 
		Let \( \eps_0>0\) be such that \(\inf_{0\le t\le n}d_{F_1}(te_1+ B_2^n)>  \eps_0 \). 
		For each \(t\), let \( x_t\) be such that \( d(x_t,F_1)>\eps_0 \).
		Let \( O_t= {\intt} \left( B(x_t, \frac{\eps_0}{2})\right) \) a ball centered at \(x_t\) with radius \( \eps_0/2\).
		Last, let \( G'\) be an open set such that 
		\( F\subset G'\) and \( G_1=G\cap \{x_{n+1}=1\}\subset F_1+\frac{\eps_0}{2}B_2^n\). 
		It's straightforward to verify that these \( \{x_t\}_t, O_t, 0\le t\le n\) and \( G'\) satisfy conditions in (\ref{ben}).
		
		\par 
		
		Now \( \bigcup_{0\le t\le n} O_t\times (a_t,1]\) is open, \( (G')^c\) is closed and 
		\[
		\cl \left(\bigcup_{0\le t\le n} O_t\times (a_t,1]\right)\subset (G')^c\subset F^c.
		\]
		Therefore, 
		\[
		\cl \left(\bigcup_{0\le t\le n} O_t\times (a_t,1]\right)\bigcap  F=\emptyset.
		\]
		In other words, \[ F\in \Fc^{	\cl \left(\bigcup_{0\le t\le n} O_t\times (a_t,1]\right)}.\]
		It's also clear that 
		\[ \Fc^{	\cl \left(\bigcup_{0\le t\le n} O_t\times (a_t,1]\right)}\bigcap \mathscr{E}_1 =\emptyset. \]

Now one sees that there is an open neighborhood of \(F\)
\[    \left(  \Fc_{(2nB_2^n)^c\times (0,1)}\bigcup \Fc^{	\cl \left(\bigcup_{0\le t\le n} O_t\times (a_t,1]\right)}\right) \bigcap  \Es_0\]
such that 
\[ F\in \left(  \Fc_{(2nB_2^n)^c\times (0,1)}\bigcup \Fc^{	\cl \left(\bigcup_{0\le t\le n} O_t\times (a_t,1]\right)}\right) \bigcap  \Es_0\subset   \mathscr{E}_0\setminus \mathscr{E}_1.  \]
We conclude that \( \Es_0\setminus \Es_1\) is open and hence  \( \mathscr{E}_1\) is closed and so is \( \mathscr{S}_1\).

\end{proof}
\par 

\begin{lemma}\label{cptcvx0c0}
	\(\mathscr{S}_{1,c}\) is closed and hence compact. 
\end{lemma}
\begin{proof}
	Suppse \( \phi_m\xrightarrow{\tau} \phi\) and \( \phi_m\in \mathscr{S}_{1,c}\) for all \(m\). We shall show that \( \phi\in \mathscr{S}_{1,c}\). In fact, 
	 Lemma \ref{compactnesss1} guarantees \( \phi\in \Ss_{1}\).  Proposition \ref{contiofcentroid} gives \( c(\phi)=0\). 
	
	
\end{proof}

\subsection{The \(\Ac\) transform and functional \( P_\Ac(\cdot) \) }

We first establish the continuity of  \( \Ac\).
\begin{proposition}\label{congoftau}
	Let \( \{ \phi_m\}\) be a sequence of functions in \( \cvx_0(\R^n)\). Then  \( \phi_m\xrightarrow{\tau} \phi\) if and only if \( \Ac[\phi_m]\xrightarrow{\tau} \Ac[\phi]\).
\end{proposition} 

\begin{proof}
	Since the transform \( \Ac\) is an involution on \( \cvx_0\) \cite{artsteinmilman2011}, it suffices to  show  one direction of the assertion, say, the ``if" part. 
	\begin{eqnarray*}
		\Ac[\phi_m]\xrightarrow{\tau} \Ac[\phi] & \Leftrightarrow &  \epi(\Ac[\phi_m])\rightarrow \epi(\Ac[\phi])
		\Leftrightarrow  
		\Rc(\epi(\phi_m)^\circ) \rightarrow \Rc(\epi(\phi)^\circ)\\
		& \Leftrightarrow & 
		\epi(\phi_m)^\circ  \rightarrow \epi(\phi)^\circ
		\Leftrightarrow  
		\epi(\phi_m)  \rightarrow \epi(\phi) 
		\Leftrightarrow  
		\phi_m \xrightarrow{\tau} \phi 
	\end{eqnarray*}
	where \( \Rc\) denotes the reflection with respect to  \(\R^n\). 
	In the second equivalence  we use the geometric description of \( \Ac\) \cite{artsteinmilman2011}. The second equivalence from bottom is verified in the Lemma \ref{convpolarset}.
\end{proof}

\par 

\begin{proposition}
	Let \( \{ \phi_m\}\) be a sequence of functions in \( \cvx_0(\R^n)\). Then  \( \phi_m\xrightarrow{\tau} \phi\) iff \( \Jc[\phi_m]\xrightarrow{\tau} \Jc[\phi]\).
\end{proposition}
\begin{proof}
	It follows straightforward from the fact that \( \Ac\circ\Lc = \Lc \circ \Ac = \Jc\).
\end{proof}
Recall that the functional 
\(P_\Ac\) is defined on convex functions in the topological space \( ( \cvx_0^+, \tau)\): 
\begin{eqnarray*}
	P_\Ac:& &  ( \cvx_0^+, \tau)\rightarrow  \R\\
	& &   \phi   \rightarrow   P_\Ac(\phi)=\int e^{-\phi}dx\int e^{-\mathcal{A}\phi}dx.   
\end{eqnarray*}
It should be understood that \( P_\Ac(\cdot)\) is only well defined on such functions in \( ( \cvx_0^+, \tau)\) that \(0\) is contained in the interior of its domain, for example, functions with centroid at 0. For even functions, \( P_\Ac(\cdot)\) is well defined by Blaschke-Santal\'o inequality (\ref{polarsantalo}). More general, for \( \phi \in \cvx_0^+\), as long as \( 0\in \intt \dom(\phi)\) we have 
\par 
\noindent
\begin{proposition}\label{integralclosedpolar}
	If \( \int e^{-\phi} dx= \infty\), then \( \int e^{-\Ac\phi} dx=0\) provided \( 0\in \intt \dom(\phi)\). If  \( \int e^{-\phi} dx>0 \), then \( \int e^{-\Ac\phi} dx>0\). Hence  \(0<\int e^{-\phi} dx<\infty  \) implies \( 0<\int e^{-\Ac\phi} dx< \infty\) provided \( 0\in \intt \dom(\phi)\).
\end{proposition} 
\par 
\noindent
\begin{proof} If \( \int e^{-\phi} dx= \infty\), then \(  G_\phi(1)\) is an unbounded set with \( 0\in \intt \dom(\phi)\).  If \( G_\phi(1)\) is bounded by a convex body \(K\), then 
\[ \phi\ge  \|\cdot\|_K\wedge I_K^\infty\]
where \(\|\cdot \|_K\wedge I_K^\infty(x)=\min\{ \|x\|_K, I_K^\infty(x)\}\). 
Hence one has 
\[   \int e^{-\phi} dx\le \int e^{-\|x\|_K}dx+\vol(K)<\infty\]
which leads to a contradiction to assumption that \( \int e^{-\phi} dx= \infty\).  Therefore, there exist \( \{K_i\}_{i=1}, K_i\subset G_\phi(1)\) with \( 0\in \intt \dom(K_i)\) for all \(i\), a volume increasing sequence of convex bodies \[ \vol(K_1)\le \vol( K_2)  \le  \cdots \]
such that \( \vol(K_i)\to \infty\) as \( i\to \infty\). 
Let \( f_i(x)= \|x\|_{K_i}\vee I_{K_i}^\infty=\max\{ \|x\|_{K_i}, \Iinfty{K_i}(x)\}, i=1,2,...\). By convexity, it is clear that \(  \phi\le f_i, \forall\ i .\) Hence one has 
\be
\Ac \phi(x) \ge   \Ac f_i(x)= \Ac\left( \|x\|_{K_i}\vee I_{K_i}^\infty \right)
=\Ac(\|x\|_{K_i})\wedge  \Ac(I_{K_i}^\infty)
= \|x\|_{K_i^\circ}\wedge I_{K_i^\circ}^\infty.
\ee
It follows that for all \(i\), 
\begin{eqnarray*} 
	\int e^{-\Ac\phi} dx   &\le & \int e^{-\|x\|_{K_i^\circ}\wedge I_{K_i^\circ}^\infty}dx
	\le \int e^{-\|x\|_{K_i^\circ}}dx+\vol(K_i^\circ)\\
	& = & \Gamma (n+1)\vol(K_i^\circ)+\vol(K_i^\circ)= (n!+1)\vol(K_i^\circ).
\end{eqnarray*}
Now notice that as \( i\to \infty\) and \( \vol(K_i)\to \infty\), \( \vol(K_i^\circ)\to 0\). Thus \( \int e^{-\Ac\phi} dx=0\). We have proved the first claim.  
\par 
If \( \int e^{-\phi}dx>0\), then the affine dimension of \( \overline{\dom(\phi)}\) is  \(n\) and so is every sub-level set \( G_\phi(s), s\neq 0\).    Hence \( \vol(G_\phi(s))>0\). Since \( 0\in \intt G_\phi(s) \), \( 0\in \intt (G_\phi(s))^\circ\) and \( \vol((G_\phi(s))^\circ)>0 \).  Then  by Lemma \ref{comparisonoflevelsets}
\begin{eqnarray*}
\int e^{-\Ac\phi} dx&=&  \int_0^1\vol(G_{\Ac\phi}(-\log s))ds\\
&\ge& \int_0^1   \vol(  (G_\phi(1/(-\log s)))^\circ) dx>0.
\end{eqnarray*} 
Therefore we conclude that   \(0<\int e^{-\phi} dx<\infty  \) implies \( 0<\int e^{-\Ac\phi} dx< \infty\) provided \( 0\in \intt \dom(\phi)\).
\end{proof}


\par 
\noindent 
\(\GL(n)\)-invariant of \( P_\Ac(\cdot)\) is well known. The proof is routine, we omit it. 
\begin{lemma} \label{PAinvariance}
	\( P_\Ac\)  is \( \GL(n)\)-invariant. 
\end{lemma}


\begin{proof}[Proof of Theorem \ref{mainthm1}]
	By \(\GL(n)\)-invariant of \( P_\Ac(\cdot)\), it suffices to prove that the functional \( P_\Ac(\cdot)\) achieves maximum and minimum on \( \Ss_e\), respectively, \( \Ss_{1,c}\). To this end,  we shall show that the functional \(P_\Ac(\cdot)\) is continuous on \( \mathscr S_e\), respectively, \( \Ss_{1,c}\). Suppose \( \{\phi_m\}_m\) is a sequence of function in \( \mathscr S_e\),  respectively, \( \Ss_{1,c}\)  such that \( \phi_m\xrightarrow{\tau} \phi\). Since \( \Ss_e\) and \(\Ss_{1,c}\) are closed, \(\phi\in \Ss_e\), respectively, \( \phi\in \Ss_{1,c}\).  By Proposition \ref{congoftau} we have \( \Ac \phi_m \xrightarrow{\tau} \Ac\phi\). By Lemma \ref{ptwconvergencedenseset}, \( \phi_m\to \phi\) pointwise on a dense set. It follows then from Lemma \ref{lemmafromAKM} that, as \(m\to \infty\), 
	\begin{eqnarray*}
		\int e^{-\phi_m}dx\raw \int e^{-\phi} dx 
	\end{eqnarray*}
	\par 
	Thus  we have 
	\[\int e^{-\phi_m}\raw \int e^{-\phi}  \text{  and   } \int e^{-\Ac \phi_m}\raw \int e^{-\Ac \phi},\]
	 as \(m\to \infty\) which in turn imply \( P_\Ac(\phi_m)\to P_\Ac(\phi)\). Now the assertions follow from the compactness of \( \mathscr{S}_e\) (Lemma \ref{cptcvx0even}), and  respectively, of  \( \Ss_{1,c}\) (Lemma \ref{cptcvx0c0}). 
\end{proof}
\par

\section{The Blaschke-Santal\'{o} Inequalities with Legendre Transform}

We  deal with convex functions with finite positive integral, i.e., 
\[ \cvx^+(\R^n)= \left\{      \phi\in \cvx(\R^n): 0<\int_{\R^n}e^{-\phi(x)}dx<\infty      \right \}  ,            \]
and the functional \(P_\Lc\) defined on functions from the space \( ( \cvx^+, \tau)\): 
\begin{eqnarray*}
    P_\Lc:& &  ( \cvx^+, \tau)\rightarrow  \R^n\\
       & &   \phi   \rightarrow   P_\Lc(\phi)=\int e^{-\phi}dx\int e^{-\mathcal{L}\phi}  dx, 
\end{eqnarray*}
where \(\Lc\) denotes the standard Legendre transform. 
One notes that for \( P_\Lc (\phi)\) to be well defined on \( \phi\), i.e., \( P_\Lc(\phi)<\infty\), we require \( 0\) is contained in the interior of \( \dom(\phi)\). For even convex functions and convex functions with centroid at 0, \( P_\Lc(\cdot)<\infty\) follows from the functional Santal\'o inequality \cite{AKM2004}\cite{FM2007}. In general, given \(\phi\) such that \( 0\in \intt(\dom(\phi))\),  one may deduce from the same idea of Lemma \ref{integralclosedpolar} that if \(0<\int_{\R^n}e^{-\phi(x)}dx<\infty \), then \( 0<\int_{\R^n}e^{-\Lc\phi(x)}dx<\infty \) (we omit proof for this claim).
To see the range of \( P_\Lc(\cdot)\) on the subset of functions in \( \cvx^+\) with centroid 0 we need also the linear invariant property of \( P_\Lc(\cdot)\). 
Define a transformation \(A=T\oplus t\) of \( \phi\in \cvx^+\), where \( T\in \GL(n), t\in \R\)  by
\[
A\phi(x)=T\phi(x)+t=\phi(Tx)+t.
\]  
As the \( \GL(n)\)-invariant for \( P_\Ac(\cdot)\), next lemma is elementary and  well known. We omit the proof.
\begin{lemma}
	\( P_\Lc\) is invariant under \( A=T\oplus t\). 
\end{lemma} 
%

\par 
Now we specify  a subset of convex functions 
\be 
\Ss_2 :=  
\bigcup_{0\le t \le  n}\{ \phi\in \cvx^+:   te_1+B_2^n\subset G_\phi(2n)\subset 2nB_2^n, \inf \phi =0, 0\le \phi(0)\le n\}.
\ee
  Let us observe  that any  convex function \( \phi\in \cvx^+\) with centroid 0 can be brought into \( \mathscr{S}_2\) by  some transformation \( A=T\oplus t\). To this end, we need a result of Fradelizi (\cite{Fradelizi1997},  see also, e.g., Theorem 10.6.2 in \cite{AGM2014}).
\begin{lemma}[\cite{Fradelizi1997}]\label{fradelizislemma}
	Let \( f= e^{-\phi}: \R^n\to [0,\infty)\) be a convex function satisfying \( \int fdx<\infty\) and \( \int xfdx=0\). Then \( f(0)\le \|f\|_\infty\le e^nf(0)\). 
\end{lemma}   
  In fact, if \( \phi\in \cvx^+\) with centroid 0, then \( \inf \phi>-\infty\). 
  Let \( \tilde{\phi}=\phi-\inf \phi\) so that \( \inf \tilde{\phi}=0\). Moreover, \( c(\tilde{\phi})=0\). Thus,  deducing from  Lemma \ref{fradelizislemma},  one has \( 0\le \tilde{\phi}(0)\le n\). 
Therefore,  \( 0\in G_{\tilde{\phi}}(2n)\). If \(0\) is on the boundary of the sub-level set \( G_{\tilde{\phi}}(2n)\) then 
  \(0\) is on the boundary of \(\dom(\tilde{\phi})\), this is not the case as \(c(\tilde{\phi})=0\). Hence \( 0\in  \intt\left(  G_{\tilde{\phi}}(2n) \right) \). By the John's theorem \ref{johnthm},  suitable linear transform \(T\) and rotation will bring \( \tilde{\phi}\) to \(\Ss_2\) (cf. (\ref{nonsymjohn})). 
\par 
\par 
\begin{lemma}\label{cptcvxc0}
\( \Ss_2\) is closed and hence compact.
\end{lemma}

\begin{proof}
	Let 
	\(
	\mathscr{E}_2=\{ F\in \Fc: F= \epi(\phi) \text{ for some } \phi\in \mathscr{S}_2\}.
	\)
Let
\begin{eqnarray*}
	& &
	\Es=\left( \bigcap_{G\in \Gc} \Fc^{G\times (-\infty,0)}\right) \bigcap\left( \bigcap_{G\in \Gc, G\subset (2nB_2^n)^c} \Fc^{G\times (-\infty,2n)}\right) \bigcap\left( \bigcap_{a\ge n}\Fc_{ \{0\}\times [n,a]}\right)\\
	& &\bigcap \{ F\in \Fc: F=\epi(\phi) ,\phi\in \sc(\R^n)\}\bigcap \{ \text{convex sets}\}.  
\end{eqnarray*}
\( \Es\) is closed since all collections in the intersection are closed in \(\tau\)-topology (cf. Proposition \ref{cvx0compact}). We will show that \( \Es\setminus \Es_2\) is open. 

	\par 
	
	Suppose that \(F\in \Es \setminus \Es_2\) and let \( \phi_F\) be such that \( \epi(\phi_F)=F\). Then \(F\) fits  in at least one of  the following cases.
	
	Case 1. \(\phi_F\)  satisfies  \( \inf \phi_F> 0 \). There exists \( 0<a<\inf \phi_F\)  such that \( F\cap (2nB_2^n\times [0,a])=\emptyset\).  
	Therefore \(       F\in \Fc^{2nB_2^n\times [0,a]}       \). It's easy to see that \( \Fc^{2nB_2^n\times [0,a]}\cap \Es_2=\emptyset\). 
	
	Case 2. \(F\cap \{x_{n+1}=1\} \subset 2nB_2^n\) does not contain any    \( {\intt} ( te_1+ B_2^n)\), \( t\in [0,n]\).  We have shown that there are \(\{x_t\in \R^n\}_{0\le t\le  n}\), \( \{0<a_t<1\}_{0\le t\le  n}\) and \( O_t\in \Gc, 0\le t\le n, G'\in \Gc \) such that 
	\( 
	x_t\in O_t,\ \ \  O_t\cap te_1+B_2^n\neq \emptyset,\ \ \ O_t\times (a_t,1]\subset (G')^c\subset F^c
	\)
	for all \( 0\le t\le n\) (cf. (\ref{ben})). 
	 
%
%
%
    \par 
    Let \[ \Oc= \left(   \Fc^{2nB_2^n\times [0,a]}  \bigcup  \Fc^{	\cl \left(\bigcup_{0\le t\le n} O_t\times (a_t,1]\right)} \right) \bigcap \Es.   \]
    Then \(\Oc\) is an open neighborhood of \(F\) that keeps \(F\) away from \( \Es_2\). 
\end{proof}

%

\begin{proof}[Proof of Theorem \ref{mainthm2}]
	The continuity of \( P_\Lc(\cdot)\) on \( \cvx^+(\R^n)\) is established in the same way with that of \( P_\Ac(\cdot)\).  Now for the first statement, note that the subset of even functions in \( \cvx^+(\R^n)\) is, up to a vertical shift, the same set as the set of even geometric convex functions. But \( P_\Lc(\cdot)\) is invariant under  linear transformations and  vertical shifts, it suffices to consider \( P_\Lc(\cdot)\) for \( \Ss_e\). 
	Therefore,  the theorem follows from the compactness of \( \mathscr{S}_e\) (Lemma \ref{cptcvx0even}). For the second statement, note that the range of \(P_\Lc \) on the set of functions in \(\cvx^+(\R^n)\) with centroid 0 is the same as that of \(P_\Lc\) on   \( \Ss_2\)  and the latter is shown to be compact (Lemma \ref{cptcvxc0}). 
\end{proof}
\par

For the functional  \( 
P_\Jc(\phi)=\frac{\int e^{-\phi}}{\int e^{-\Jc \phi}} \)
defined on \( \cvx_{0}^+\),
it is readily seen that \( P_\Jc(\cdot)\) is invariant under linear transform \( T\in \GL(n)\). In fact, \( \Jc(T\phi)= \Lc\Ac(T\phi)= \Lc((T^{-1})^t(\Ac\phi))=T(\Lc\Ac\phi)=T(\Jc\phi)\).
Thus along the same line with the case for \( \Ac\) transform, we have the following theorem which is also easily inferred from the results in \cite{alexdan}.

\begin{theorem}
			The functional \( P_\Jc(\cdot)\) achieves maximum and minimum on the subset of even functions in \( \cvx_0^+(\R^n)\).
		It also achieves maximum and minimum on the subset of functions in \( \cvx_0^+(\R^n)\) with centroid 0.	
\end{theorem}
We omit the proof.

\vspace{.1in}
\bibliographystyle{amsplain} 
%


\bibliography{refconvexanalysis}

\end{document}